\documentclass[11pt]{amsart}

\usepackage{geometry}
\usepackage{epsfig}
\usepackage{graphicx}
\usepackage{amsbsy}
\usepackage{amsmath}
\usepackage{amsfonts}
\usepackage{amssymb}
\usepackage{textcomp}
\usepackage{aliascnt}

\usepackage{xcolor} 
\colorlet{darkishRed}{red!80!black}
\colorlet{darkishBlue}{blue!60!black}
\colorlet{darkishGreen}{green!60!black}
\usepackage{hyperref}
\hypersetup{
	colorlinks,
	linkcolor={red!60!black},
	citecolor={green!60!black},
	urlcolor={blue!60!black},
}
\usepackage[nameinlink, capitalise, noabbrev]{cleveref}
\crefformat{enumi}{#2#1#3}
\crefformat{equation}{#2(#1)#3}
\crefname{mainresult}{Theorem}{Theorems}

\makeatletter
\def\calCommandfactory#1{%
  \expandafter\def\csname c#1\endcsname{\mathcal{#1}}}
\def\frakCommandfactory#1{%
  \expandafter\def\csname frak#1\endcsname{\mathfrak{#1}}}
\newcounter{ctr}
\loop
  \stepcounter{ctr}
  \edef\X{\@Alph\c@ctr}
  \expandafter\calCommandfactory\X
  \expandafter\frakCommandfactory\X
\ifnum\thectr<26
\repeat

\newcommand{\Sbb}{\mathbb{S}}
\newcommand{\Nbb}{\mathbb{N}}
\newcommand{\Zbb}{\mathbb{Z}}

{\begin{array}{c}
\setlength{\unitlength}{1em}}%
{\end{array}}

\usepackage{amsthm}

\numberwithin{equation}{section}

\newcommand{\theoremize}[2]{\newaliascnt{#1}{thm} \newtheorem{#1}[#1]{#2} \aliascntresetthe{#1}}

\theoremstyle{plain}
\newtheorem{mainresult}{Theorem}

\newtheorem{innercustomque}{Question}
\newtheorem{thm}{Theorem}[section]
\theoremize{lem}{Lemma}
\theoremize{skolem}{Skolem}
\theoremize{fact}{Fact}
\theoremize{sublem}{Sublemma}
\theoremize{claim}{Claim}
\theoremize{obs}{Observation}
\theoremize{prop}{Proposition}
\theoremize{cor}{Corollary}
\theoremize{que}{Question}
\theoremize{oque}{Open Question}
\theoremize{con}{Conjecture}

\theoremstyle{definition}
\theoremize{dfn}{Definition}
\theoremize{rem}{Remark}
\theoremize{eg}{Example}
\theoremize{exercise}{Exercise}
\theoremstyle{plain}

\newenvironment{claimproof}{
  \pushQED{\qed}%
  \trivlist
  \item[\hskip\labelsep
        \itshape
    Proof of Claim.]\ignorespaces
}{%
  
  \popQED\endtrivlist
}

\newenvironment{customque}[1]
  {\renewcommand\theinnercustomque{#1}\innercustomque}
  {\endinnercustomque}
\usepackage{verbatim}
\usepackage{enumerate}
\usepackage[all]{xy}
\usepackage{bbm}

\usepackage{subfig}

\title{Colouring the 1-skeleton of $d$-dimensional triangulations}

\author{Tim Planken}
\address{TU Bergakademie Freiberg}

\newcommand{\unit}{\mathbbm{1}}

\makeatletter
\DeclareRobustCommand{\cev}[1]{%
  {\mathpalette\do@cev{#1}}%
}
\newcommand{\do@cev}[2]{%
  \vbox{\offinterlineskip
    \sbox\z@{$\m@th#1 x$}%
    \ialign{##\cr
      \hidewidth\reflectbox{$\m@th#1\vec{}\mkern4mu$}\hidewidth\cr
      \noalign{\kern-\ht\z@}
      $\m@th#1#2$\cr
    }%
  }%
}
\makeatother

\begin{document}

\begin{abstract}
While every plane triangulation is colourable with three or four colours, Heawood showed that a plane triangulation is 3-colourable if and only if every vertex has even degree.
In $d \geq 3$ dimensions, however, every $k \geq d+1$ may occur as the chromatic number of some triangulation of $\Sbb^d$.
As a first step, Joswig~\cite{Joswig_2002} structurally characterised which triangulations of $\Sbb^d$ have a $(d+1)$-colourable 1-skeleton. In the 20 years since Joswig's result, no characterisations have been found for any $k>d+1$.

In this paper, we structurally characterise which triangulations of $\Sbb^d$ have a $(d+2)$-colourable 1-skeleton: they are precisely the triangulations that have a subdivision such that for every $(d-2)$-cell, the number of incident $(d-1)$-cells is divisible by three.
\end{abstract}

\maketitle

\section{Introduction}
\label{sec:introduction}

We consider the problem of colouring the vertices of the 1-skeleton\footnote{We define 1-skeleton, chamber and triangulation in \autoref{sec:definitions}.} of triangulations of the $d$-dimensional sphere $\Sbb^d$ for $d \geq 2$.
For each such triangulation, we need at least $d+1$ colours since every chamber induces a complete graph on $d+1$ vertices in the 1-skeleton.

For $d=2$, Heawood~\cite{heawood1898four} showed that a plane triangulation is colourable with three colours if and only if every vertex has even degree.
See also~\cite{golovina2019induction, krol1972sufficient, lovasz2007combinatorial, Tsai_2011} for alternative proofs and variations of Heawood's theorem.
On the other hand, by the four-colour theorem~\cite{Appel_1989}, \emph{every} plane triangulation is colourable with four colours.

For $d \geq 3$, however, there exists for every $k \geq 1$ a triangulation of $\Sbb^d$ whose 1-skeleton contains the complete graph $K_{d+k}$ (\cite{ueckerdtpersonalcommunication} and \autoref{thm:counterexample}).
Naturally, this raises the following question:

\begin{que}
\label{que:1}
    Let $k \geq 1$ be an integer.
    Can you find a structural characterisation for all $d \geq 3$ of the triangulations of $\Sbb^d$ whose $1$-skeleton is $(d+k)$-colourable?
\end{que}

Joswig~\cite{Joswig_2002} answered \autoref{que:1} for $k=1$ as follows:

\begin{thm}[Joswig~\cite{Joswig_2002}, `moreover'-part by Carmesin, Nevinson and Saunders~\cite{carmesin2022characterisation}]
    \label{thm:heawood}
    Let $d \geq 2$ be an integer and let $C$ be a triangulation of $\Sbb^d$.
    Then the following assertions are equivalent.
    \begin{enumerate}
        \item\label{itm:main-heawood-1} The $1$-skeleton of $C$ has a proper $(d+1)$-colouring.
        \item\label{itm:main-heawood-2} All $(d-2)$-cells of $C$ are incident with an even number of $(d-1)$-cells.
    \end{enumerate}
    Moreover, if $d=3$, then we may add:
    \begin{enumerate}
        \setcounter{enumi}{2}
        \item\label{itm:main-heawood-3} There exists a 3-edge-colouring of $C$ such that every 2-cell contains edges of all colours.
    \end{enumerate}
\end{thm}

In the 20 years since Joswig's result, no characterisations have been found for any $k>1$.
Heawood~\cite{heawood1898four} observed that the four-colour theorem is equivalent to the statement that every plane triangulation $G$ has a subdivision\footnote{A plane triangulation $G'$ is a \emph{subdivision} of a plane triangulation $G$ if $G'$ is obtained from $G$ by adding a vertex $v_f$ in some faces $f$ of $G$ and joining each $v_f$ to all vertices in the boundary of $f$.} $G'$ such that all vertices in $G'$ have degree divisible by three, see \autoref{fig:example-subdivision}.
Carmesin~\cite{carmesinpersonalcommunication} asked for a characterisation of the triangulations of $\Sbb^3$ that admit subdivisions\footnote{For the definition of a \emph{subdivision} of a triangulation of $\Sbb^d$ for $d \geq 3$ see \autoref{sec:subdivisions}.} such that the number of $2$-cells incident with each edge is divisible by three.
We answer his question, even more generally in arbitrary dimensions, and thereby also solve \autoref{que:1} for $k=2$.
\begin{figure}
    \centering
    \includegraphics[width=0.7\textwidth]{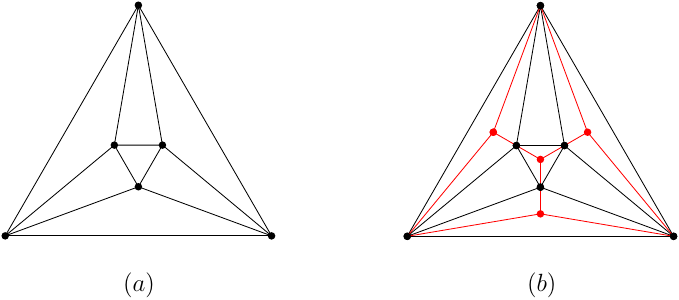}
    \caption{(a) A plane triangulation $G$ and (b) a subdivision $G'$ of $G$ such that all vertices have degree divisible by three. The added vertices and edges are red.}
    \label{fig:example-subdivision}
\end{figure}

\begin{mainresult}
    \label{thm:subdivisions}
    Let $d \geq 2$ be an integer and let $C$ be a triangulation of $\mathbb{S}^d$.
    Then the following assertions are equivalent.
    \begin{enumerate}
        \item\label{itm:main-subdivisions-1} The $1$-skeleton of $C$ has a proper $(d+2)$-colouring.
        \item\label{itm:main-subdivisions-2} There exists a subdivision $C'$ of $C$ such that for every $(d-2)$-cell, the number of incident $(d-1)$-cells is divisible by three.
    \end{enumerate}
    Moreover, if $d=3$, then we may add:
    \begin{enumerate}
        \setcounter{enumi}{2}
        \item\label{itm:main-subdivisions-3} The maximal subdivision of $C$ has a 2-edge-colouring such that every 2-cell  contains edges of both colours.
    \end{enumerate}
\end{mainresult}

In contrast to Joswig's result, that allows to construct a proper $(d+1)$-colouring for any given triangulation of $\Sbb^d$ in polynomial time, we show that deciding whether the 1-skeleton of a given triangulation of $\Sbb^d$ is $(d+2)$-colourable is $NP$-complete for each $d \geq 3$.

To prove \autoref{thm:subdivisions}, I independently found a stronger version of Joswig's method, in the form of the \nameref{lem:local-global-colouring}, see \autoref{lem:local-global-colouring} below.
Indeed, we use the \nameref{lem:local-global-colouring} to prove both \autoref{thm:subdivisions} and \autoref{thm:heawood}.

Roughly, the idea of the \nameref{lem:local-global-colouring} is as follows. Let $C$ be a triangulation of $\Sbb^d$, whose vertices we want to colour, and let $s_0$ be an arbitrary chamber of $C$.
We start by assigning distinct colours to the vertices on the boundary of $s_0$. Now suppose that for every two chambers $s$ and $t$ of $C$ sharing a $(d-1)$-cell we are given a map $g_{\vec{st}}$ that determines the colours of the vertices on the boundary of~$t$, given any colouring of the vertices on the boundary of~$s$. Then we greedily extend the colouring of the vertices on the boundary of $s_0$ to a colouring $c$ of $C$ via the maps $g_{\vec{st}}$.

The \nameref{lem:local-global-colouring} states that the colouring $c$ is well-defined and proper if the maps $g_{\vec{st}}$ are compatible in the following sense: 
On the one hand, each map $g_{\vec{st}}$ must fix the colours of the vertices of the $(d-1)$-cell shared by $s$ and $t$. On the other hand, cyclically going around a $(d-2)$-cell with the maps $g_{\vec{st}}$ yields the identity.
A collection of maps $g_{\vec{st}}$ that satisfy both conditions is called a \emph{proper canonical local colouring}, see \autoref{sec:local-global}.

\begin{lem}[Local-Global Colouring Lemma]
\label{lem:local-global-colouring}
    Let $d \geq 2$ be an integer.
    A triangulation of $\Sbb^d$ is canonically locally $k$-colourable if and only if its 1-skeleton is $k$-colourable.
\end{lem}

\subsection{Related work}
There exists a rich literature on extensions of theorems from two dimensions to three dimensions.
For example, Carmesin~\cite{carmesin2023embedding} proved a 3-dimensional analogue of Kuratowski's theorem.
Carmesin and Mihaylov~\cite{carmesin2021outerspatial} extended the concept and excluded-minors characterisation of outerplanar graphs.
Kurkofka and Nevinson~\cite{Kurkofka_2025} asked for the least integer $k$ such that every simplicial $2$-complex embedded in $\Sbb^3$ has a $k$-edge-colouring, and showed $k \leq 12$.
Georgakopoulos and Kim~\cite{Georgakopoulos_2022} extended Whitney’s Theorem on unique embeddings of 3-connected planar graphs.

\subsection{Organization of the paper}
In \autoref{sec:definitions} we introduce the necessary definitions and terminology that will be used in this paper.
In \autoref{sec:local-global} we prove the \nameref{lem:local-global-colouring} (\ref{lem:local-global-colouring}).
In \autoref{sec:heawood} we prove (\ref{itm:main-heawood-1}) $\Leftrightarrow$ (\ref{itm:main-heawood-2}) of \autoref{thm:heawood}.
In \autoref{sec:subdivisions} we prove (\ref{itm:main-subdivisions-1}) $\Leftrightarrow$ (\ref{itm:main-subdivisions-2}) of \autoref{thm:subdivisions} and show that for each $d \geq 3$ it is $NP$-complete to decide whether the 1-skeleton of a given triangulation of~$\Sbb^d$ is $(d+2)$-colourable.
In \autoref{sec:2-edge-colouring} we prove the `moreover'-part of \autoref{thm:subdivisions}.
In \autoref{sec:edge-colourings}, we show how to obtain edge-colourings and face-colourings from vertex-colourings of triangulations of $\Sbb^d$.
Then, we conclude the paper with open problems in \autoref{sec:outlook}.

\section{Definitions and Terminology}
\label{sec:definitions}

\subsection{Topology}
\label{subsec:topology}

For the background in algebraic topology, we refer to Hatcher's book~\cite{hatcheralgebraic}.

We assume familiarity with~\cite[Sections 2.4, 2.5 and 4.1]{Diestel_2022}.
In particular, let $G$ be a graph and $\cC$ be a set of cycles in $G$.
If $v$ is a vertex in $G$ then $\pi_1^\cC(G,v)$ is the subgroup of $\pi_1(G,v)$ generated by cycles in $\cC$.
Moreover, $\cC$ \emph{generates} a cycle $c$ in $G$ if there exists a vertex $v$ in $G$ such that some walk in $\pi_1^{\cC}(G,v)$ stems from $c$.

Given a cell complex $C$ embedded in $\Sbb^d$, a \emph{chamber} of $C$ is a connected component of $\Sbb^d \setminus C$.
A \emph{triangulation of $\Sbb^d$} is a simplicial $(d-1)$-complex embedded in $\Sbb^d$ such that all chambers are $d$-simplices.
In this paper we only consider triangulations of $\Sbb^d$ with $d \geq 2$.
Throughout this paper we assume that every triangulation of $\Sbb^d$ comes with an ordering $<$ of its vertices.
The \emph{$d'$-skeleton} of a cell complex $C$ is the cell complex $C'$ consisting of all $i$-cells in $C$ with $i \leq d'$.

We say that a vertex $v$ is a vertex of the chamber $t$ if it lies in the boundary of $t$.
Analogously, we say that $v$ is a vertex of a cell $f$ if $v$ is contained in $f$.
Two chambers $s$ and $t$ \emph{share} the $d'$-cell~$f$ if~$f$ is included in the boundary of both $s$ and $t$.
If a chamber $t$ has vertices $v_1,v_2,\ldots,v_k$ on its boundary, we say that $t$ is a chamber \emph{on the vertices} $v_1,v_2,\ldots,v_k$.

Let $C$ be a triangulation of $\Sbb^d$ and $\tilde C$ denote the cell complex embedded in $\Sbb^d$ that is obtained from $C$ by adding all chambers as $d$-simplices.
For the definition of the \emph{dual cell complex~$C^*$} of~$\tilde C$ see~\cite[Chapter X]{Seifert_1980} and~\cite[Section 3.3]{hatcheralgebraic}.
In particular, $C^*$ is again embedded in $\Sbb^d$ and every $d'$-cell of $\tilde C$ (for $0 \leq d' \leq d$) corresponds to a $(d-d')$-cell of $C^*$.
Given a triangulation $C$ of $\Sbb^d$, the \emph{dual 2-complex} of $C$ is the 2-skeleton of the dual cell complex of $C$ and the \emph{dual graph} of $C$ is the 1-skeleton of the dual cell complex of $C$.
Note that the dual graph of $C$ may have parallel edges.

Observe that every $(d-1)$-cell in a triangulation $C$ of $\Sbb^d$ is included in the boundary of exactly two chambers.
Let $f$ be a $(d-2)$-cell in $C$ and $c = t_0 f_1 t_1 f_2 t_2 \dots t_{k-1} f_k t_k$ be the cyclic ordering of the $(d-1)$-cells $f_1,\dots,f_k$ and chambers $t_0,t_1,\dots,t_k=t_0$ around $f$ induced by the embedding of the triangulation $C$ in $\mathbb{S}^d$.
Then $c$ is a cycle in the dual graph $G$ of $C$, which we call the \emph{$f$-cycle in $G$}.

The dual graph $G$ of a triangulation $C$ of $\Sbb^d$ can also be constructed as follows.
The set $V(G)$ of vertices is the set of chambers of $C$.
For each $(d-1)$-cell that lies on the boundary of the two chambers $s$ and $t$ of $C$ we add an edge $st$ to $G$.
The dual 2-complex $D$ of $C$ can be obtained from the dual graph $G$ of $C$ by adding all $f$-cycles (for all $(d-2)$-cells $f$ in $C$) as 2-cells to $D$.

It is well-known that the $d$-dimensional sphere $\Sbb^d$ is orientable.
We assume throughout this paper that we are given $\Sbb^d$ together with an orientation.
Let $C$ be a triangulation of $\Sbb^d$ and $t$ be a chamber of $C$.
The orientation of $\Sbb^d$ induces an orientation on every $d$-simplex embedded in $\Sbb^d$ and therefore on the chamber $t$ (which we can also regard as a $d$-simplex).
This orientation of $t$ determines a vertex ordering $v_0,v_1,\dots,v_d$ of the vertices in the boundary of $t$ up to a permutation of even parity.
We say that such a vertex ordering determined by the orientation of $\Sbb^d$ is \emph{positive}, and all other vertex orderings of $t$ are called \emph{negative}.
See~\cite{hatcheralgebraic, Munkres_2018} for more details.

\subsection{Gain graphs}

Throughout this paper we assume that all directed graphs have no loops.
For a directed graph $G$ and a pair $\vec e = uv$ of vertices, we define $\cev e = vu$.
A directed graph $G$ is \emph{symmetric} if for every pair $\vec e = uv$ of vertices, $\vec e \in E(G)$ if and only if $\cev e \in E(G)$.
Let $(\Gamma, \cdot)$ be a finite group with neutral element $\unit$.
A \emph{gain graph} is a symmetric directed graph G together with an assignment of weights $a_{\Vec{e}}\in \Gamma$ to every directed edge $\Vec{e}$ such that $a_{\cev{e}} = a_{\Vec{e}}^{-1}$.

Let $G$ be a gain graph and $W = v_0 \vec e_1 v_1 \vec e_2 v_2 \dots v_{k-1} \vec e_k v_k$ be a walk in $G$ with $\vec e_i = v_{i-1} v_{i}$ all $i$.
The \emph{gain value $\ell_W$} of the walk $W$ is defined by $\ell_W = a_{\vec e_k} \dots a_{\vec e_2} a_{\vec e_1} \in \Gamma$.
Note that, in general, the gain value of a closed walk depends on the start point and the direction of the closed walk.
A closed walk $c$ is \emph{balanced} if $\ell_c = \unit$, which is independent of the start point and direction of the closed walk.

\section{The Local-Global Colouring Lemma}
\label{sec:local-global}

In this section, we prove the \nameref{lem:local-global-colouring} (\autoref{lem:local-global-colouring}).

\begin{fact}
\label{fact:generating-cycles-gain-value}
    Let $\Gamma$ be a group.
    Let $G$ be a gain graph with weights in $\Gamma$ and $\cC$ be a set of cycles of $G$ that generate all cycles of $G$.
    If every cycle in $\cC$ is balanced, then every cycle in $G$ is balanced. \qed
\end{fact}

\begin{lem}
\label{lem:dual-graph-generating-cycles}
    Let $G$ be the dual graph of a triangulation $C$ of $\Sbb^d$.
    Then the set of all $f$-cycles in~$G$ (for $(d-2)$-cells $f$ in $C$) generates all cycles in $G$.
\end{lem}
\begin{proof}
    Let $\tilde C$ be the simplicial $d$-complex that is obtained from $C$ by adding all chambers as $d$-simplices.
    Let $C^*$ be the dual cell complex of $\tilde C$ and recall that the $2$-skeleton of $C^*$ is the dual $2$-complex $D$ of $C$.
    Observe that $C^*$ is simply connected.
    Then it follows that the $2$-skeleton of~$C^*$, i.e.\ the dual $2$-complex $D$, is again simply connected, see Hatcher~\cite[Proposition 1.26 (b)]{hatcheralgebraic}.
    By construction, the face cycles (i.e.\ the boundaries of the 2-cells) of $D$ are exactly the cycles in~$\cC$.
    Since $D$ is simply connected, the face cycles of $D$ generate all cycles in the 1-skeleton of $D$ (see~\cite[Proposition 1.26 (a)]{hatcheralgebraic}) and therefore in $G$.
\end{proof}

The following corollary follows immediately from \autoref{fact:generating-cycles-gain-value} and \autoref{lem:dual-graph-generating-cycles}.
\begin{cor}
\label{cor:group-cycle-gain-value-dual-graph}
    Let $G$ be a dual gain graph of a triangulation of $\mathbb{S}^d$.
    If every $f$-cycle in $G$ is balanced, then every cycle in $G$ is balanced. \qed
\end{cor}

For a finite set $X$, we define an \emph{$X$-gain graph} to be a gain graph where the weights on the edges are permutations of $X$, i.e.\ for every directed edge $\vec e \in E(G)$ there exists a bijective map $g_{\vec{e}} \colon X \to X$ such that $g_{\cev e} = g_{\vec e}^{-1}$.
Again, we say that an $X$-gain graph is a \emph{dual $X$-gain graph} of a triangulation $C$ of $\Sbb^d$ if the corresponding undirected graph is the dual graph of $C$.

Let $G$ be an $X$-gain graph with bijective maps $g_{\vec e} \colon X \to X$.
We say that a colouring $\phi \colon V(G) \to X$ of the vertices of $G$ with elements of $X$ \emph{commutes} with $g_{\vec e}$ for some edge $\vec e = uv$ in $G$ if $g_{\vec e}(\phi(u)) = \phi(v)$ holds.
Observe that if $\phi$ commutes with $g_{\vec e}$, it also commutes with $g_{\cev e}$ since $g_{\cev e}(\phi(v)) = g_{\vec e}^{-1}(\phi(v)) = \phi(u)$.

\begin{lem}
\label{lem:colouring-respects-edges}
    Let $G$ be an $X$-gain graph with maps $g_{\vec e}$ assigned to the directed edges $\vec e$.
    Let $\cC$ be a set of cycles of $G$ that generate all cycles in $G$.
    If every cycle in $\cC$ is balanced, then there exists a colouring $\phi \colon V(G) \to X$ of $G$ with elements of $X$ that commutes with all maps $g_{\vec e}$.
\end{lem}
\begin{proof}
    Since the cycles in $\cC$ generate all cycles in $G$, we have that all cycles in $G$ are balanced by \autoref{fact:generating-cycles-gain-value}.
    Let $u \in V(G)$ be an arbitrary vertex in $G$ to which we assign an arbitrary colour $\phi(u) \in X$.
    Let $T$ be a spanning tree of $G$.
    For each vertex $v$ in $G$, let $P_v = u \vec e_1 v_1 \dots v_{k-1} \vec e_k v$ be the path in $T$ from $u$ to $v$.
    We define the colour of $v$ to be $\phi(v) = (g_{\vec e_k} \circ \dots \circ g_{\vec e_1})(\phi(u))$.
    Clearly, this colouring commutes with all maps $g_{\vec e}$ where $\vec e$ is an edge of the spanning tree $T$.
    For an edge $\vec{e'} = w_1 w_2 \in E(G) \setminus E(T)$, let $P$ be the path from $w_1$ to $w_2$ in $T$.
    Since the cycle $c = w_1 P w_2 \cev{e'} w_1$ is balanced (i.e.\ has gain value $g_c = \unit$) and $\phi$ commutes with all $g_{\vec e}$ for edges $\vec e$ in $T$, it also commutes with $g_{\vec{e'}}$.
    Therefore, the colouring $\phi$ of the vertices of $G$ commutes with all $g_{\vec e}$.
\end{proof}

Now, we formally define what it means for a triangulation of $\Sbb^d$ to be canonically locally $k$-colourable.
Recall that we assume that every triangulation of $\Sbb^d$ comes with an ordering $<$ of its vertices.
We denote by $\Sigma_{k}$ the symmetric group whose elements are the permutations on $\{0,1,\dots,k-1\}$.
Let $C$ be a triangulation of $\Sbb^d$ and let $k \geq d+1$ be an integer.
Given a chamber~$t$ of~$C$ and a permutation $\pi \in \Sigma_{k}$, we call $(t,\pi)$ \emph{colouring} of $t$.
If the chamber $t$ has vertices $u_0 < u_1 < \dots < u_d$ in its boundary, the colour of $u_i$ \emph{induced} by $(t, \pi)$ is $\pi(i)$.
Given two chambers $t$ and $t'$ incident to the same $d'$-cell~$f'$ and two permutations~$\pi,\pi' \in \Sigma_{k}$, we say that the colourings $(t,\pi)$ and $(t',\pi')$ \emph{agree} on $f'$ if, for each vertex $u$ of $f'$, the colour of $u$ induced by $(t,\pi)$ is equal to the colour of $u$ induced by $(t',\pi)$.
\begin{dfn}[Canonically Locally Colourable]
\label{dfn:locally-colourable}
    Let $C$ be a triangulation of $\Sbb^d$ and $k \geq d+1$.
    A \emph{proper canonical local $k$-colouring} of $C$ is a dual $\Sigma_{k}$-gain graph $G$ of $C$ with bijective maps $g_{\vec e} \colon \Sigma_{k} \to \Sigma_{k}$ on the directed edges of $G$ such that the following two conditions hold.
    \begin{enumerate}[(1)]
        \item\label{itm:locally-colourable-1} For every $(d-2)$-cell $f$ in $C$, the $f$-cycle $c$ in $G$ is balanced.
        \item\label{itm:locally-colourable-2} For every $(d-1)$-cell $f'$ in $C$ and corresponding edge $\vec e = st$ in $G$, each colouring $(s,\pi_s) \in \Sigma_{k}$ agrees with the colouring $(t,g_{\vec e}(\pi_s))$ on $f'$.
    \end{enumerate}
    We say that $C$ is \emph{canonically locally $k$-colourable} if there exists a proper canonical local $k$-colouring of $C$.
\end{dfn}

\begin{proof}[Proof of the \nameref{lem:local-global-colouring} (\autoref{lem:local-global-colouring})]
    ($\Rightarrow$)
    Let $C$ be a triangulation of $\Sbb^d$ that is canonically locally $k$-colourable.
    That is, there exists a dual $\Sigma_{k}$-gain graph $G$ with bijective maps $g_{\vec e} \colon \Sigma_{k} \to \Sigma_{k}$ on the directed edges $\vec e$ that satisfy \autoref{dfn:locally-colourable} (\ref{itm:locally-colourable-1}) and (\ref{itm:locally-colourable-2}).
    We need to show that the 1-skeleton of $C$ has a proper $k$-colouring.
    
    By \autoref{lem:dual-graph-generating-cycles} and \autoref{lem:colouring-respects-edges}, there exists a colouring of the vertices of $G$ with elements from~$\Sigma_{k}$ that commutes with all maps $g_{\vec e}$.
    We denote by $\pi_t \in \Sigma_{k}$ the colour assigned to the vertex~$t$ of~$G$.

    In the first step, we show that for each vertex $u$ of $C$, the colourings $(t,\pi_t)$ of all incident chambers $t$ agree on $u$.
    Let $s$ and $t$ be two chambers incident to $u$.
    Then there exists a path $t_0 t_1 t_2 \dots t_{m-1} t_m$ in the dual graph $G$ with $t_0 = s$ and $t_m = t$ such that the chambers $t_0,t_1,\dots,t_k$ are all incident to $u$.
    By \autoref{dfn:locally-colourable} (\ref{itm:locally-colourable-2}), the colourings $(t_i,\pi_{t_i})$ and $(t_{i+1},\pi_{t_{i+1}})$ of two adjacent chambers $t_i$ and $t_{i+1}$ agree on the colour of $u$.
    Hence, all colourings $(t_i, \pi_{t_i})$ agree on the colour of $u$.
    It follows that any two chambers incident to $u$ agree on the colour of $u$.

    Then we define the colour of $u$ to be the colour of $u$ induced by the colourings of the incident chambers.
    This is a colouring of the $1$-skeleton with elements in $\{0,1,\dots,k-1\}$.
    To see that this colouring is proper, let $u_1 u_2$ be an edge in the $1$-skeleton of $C$.
    Then there exists a chamber $s$ that is incident to both $u_1$ and $u_2$.
    Since $s$ is coloured with an element $\pi_s \in \Sigma_{k}$, it assigns distinct colours to $u_1$ and $u_2$.

    ($\Leftarrow$)
    Let $\psi \colon V \to \{0,1,\dots,k-1\}$ be a proper $k$-colouring of the 1-skeleton of $C$.
    For each chamber $t$ of $C$, we fix a colouring $(t,\pi_t)$ with $\pi_t \in \Sigma_{k}$ such that for each vertex $v$ in the boundary of $t$, the colour of $u$ induced by $(t,\pi_t)$ is $\psi(u)$.
    To construct a proper canonical local $(k+1)$-colouring of $C$, we have to define a dual $\Sigma_{k}$-gain graph of $C$.
    For each edge $\vec e = st$ in the dual graph $G$ and for each $\sigma_s \in \Sigma_{k}$, we define $g_{\vec e}(\sigma_s) = \sigma_s \circ \pi_s^{-1} \circ \pi_t$.
    
    Obviously, $g_{\vec e}$ is a bijective map.
    Moreover, observe that if $\sigma_t = g_{\vec e}(\sigma_s)$ then $g_{\cev e}(\sigma_t) = \sigma_t \circ \pi_t^{-1} \circ \pi_s = \sigma_s$ since $\sigma_t \circ \pi^{-1}_t = \sigma_s \circ \pi^{-1}_s$.
    This proves $g_{\vec e} = g_{\cev e}^{-1}$ for all edges $\vec e$.
    It remains to check that \autoref{dfn:locally-colourable} (\ref{itm:locally-colourable-1}) and (\ref{itm:locally-colourable-2}) are satisfied.
    
    (\ref{itm:locally-colourable-1}) Let $f$ be a $(d-2)$-cell in $C$ with corresponding $f$-cycle $c = t_0 \vec e_1 t_1 \dots t_{\ell-1} \vec e_{\ell - 1} t_0$ in the dual graph $G$.
    Let $\sigma_{t_0} \in \Sigma_{k}$ be arbitrary.
    Then $g_c(\sigma_{t_0}) = (g_{\vec e_{\ell - 1}} \circ \cdots \circ g_{\vec e_1})(\sigma_{t_0}) = \sigma_{t_0} \circ (\pi_{t_0}^{-1} \circ \pi_{t_1}) \circ (\pi_{t_1}^{-1} \circ \pi_{t_2}) \circ \cdots \circ (\pi_{t_{\ell-1}}^{-1} \circ \pi_{t_0}) = \sigma_{t_0}$.
    
    (\ref{itm:locally-colourable-2}) Let $s$ and $t$ be two chambers in $C$ that share a $(d-1)$-cell $f$ with corresponding edge $\vec e = st$ in the dual graph $G$.
    Let $u$ be an arbitrary vertex of $f$ with index $i$ in the ordering of the vertices of $s$ and with index $j$ in the ordering of the vertices of $t$.
    Observe that $\pi_s(i) = \pi_t(j) = \psi(u)$.
    Let $\sigma_s \in \Sigma_{k}$ be arbitrary.
    By definition, the colouring $(s,\sigma_s)$ of $s$ induces the colour $\sigma_s(i)$ on $u$.
    Then,
    \[
        (g_{\vec e}(\sigma_s))(j) = (\sigma_s \circ \pi_s^{-1} \circ \pi_t)(j) = \sigma_s(\pi_s^{-1}(\pi_t(j))) = \sigma_s(\pi_s^{-1}(\pi_s(i))) = \sigma_s(i)~,
    \]
    which proves that $(s,\sigma_s)$ and $(t,\sigma_t = g_{\vec e}(\pi'_s))$ induce the same colour on the vertex $u$.
\end{proof}

\section{Proof of (\ref{itm:main-heawood-1}) \texorpdfstring{$\Leftrightarrow$}{<=>} (\ref{itm:main-heawood-2}) of \autoref{thm:heawood}}
\label{sec:heawood}

For convenience, we include a proof of Joswig's \autoref{thm:heawood}. Readers familiar with~\cite{Joswig_2002} are encouraged to skip to \autoref{sec:subdivisions}.

Let $C$ be a triangulation of $\Sbb^d$ and recall that we assume that we are given $\Sbb^d$ together with a fixed orientation.
Recall that this induces a fixed orientation on every chamber $t$ of $C$.
Moreover, recall that this defines positive and negative vertex orderings of the vertices in the boundary of $t$.
We can use this to define the orientation of a chamber with respect to a given $(d+1)$-colouring~$\psi$ as follows.

\begin{dfn}[Orientation of properly $(d+1)$-coloured chambers]
    Let $C$ be a triangulation of~$\Sbb^d$ with a proper $(d+1)$-colouring $\psi \colon V \to \{0,\dots,d\}$ of its 1-skeleton.
    For a chamber $t$ of $C$, let $v_0, \dots, v_d$ be a positive vertex-ordering.
    Then the $\psi$-orientation of $t$ is \emph{positive} if $i \mapsto \psi(i)$ is an even permutation,
    and otherwise \emph{negative}.
\end{dfn}

Observe that for a proper $(d+1)$-colouring $\psi$, a chamber $t$ has positive $\psi$-orientation if and only if $\psi^{-1}(0), \psi^{-1}(1),\dots,\psi^{-1}(d)$ is a positive vertex-ordering.
Moreover, note that whether the $\psi$-orientation of a chamber $t$ is positive/negative does \emph{not} depend on the chosen vertex-ordering $v_0, \dots, v_d$ of $t$, as long as the vertex-ordering is positive.
In particular, if $v_0, \dots, v_d$ is a negative vertex-ordering, then the parity of the permutation $\left(\begin{smallmatrix}0 & 1 & \cdots & d \\ \psi(v_0) & \psi(v_1) & \cdots & \psi(v_d) \end{smallmatrix}\right)$ is even if and only if the $\psi$-orientation of $t$ is negative.
Analogously, choosing a reverse orientation of $\Sbb^d$ flips all $\psi$-orientations of chambers.

In order to prove (\ref{itm:main-heawood-1}) $\Leftrightarrow$ (\ref{itm:main-heawood-2}) of \autoref{thm:heawood}, we show the following strengthening.

\begin{thm}[Joswig~\cite{Joswig_2002}]
    Let $C$ be a triangulation of $\Sbb^d$.
    Then the following statements are equivalent.
    \begin{enumerate}[(1)]
        \item\label{itm:heawood-1} The 1-skeleton of $C$ is $(d+1)$-colourable.
        \item\label{itm:heawood-2} There exists a proper $(d+1)$-colouring $\psi$ of the 1-skeleton of $C$ such that for every chamber $t$ of $C$, every chamber that shares a $(d-1)$-cell with $t$ in $C$ has the reverse $\psi$-orientation to $t$.
        \item\label{itm:heawood-3} The dual graph of $C$ is bipartite.
        \item\label{itm:heawood-4} All $(d-2)$-cells of $C$ are incident with an even number of $(d-1)$-cells.
    \end{enumerate}
\end{thm}
\begin{proof}
    (\ref{itm:heawood-1}) $\Rightarrow$ (\ref{itm:heawood-2}).
    Let $C$ be a triangulation of $\Sbb^d$ with a proper $(d+1)$-colouring $\psi \colon V \to \{0,1,\dots,d\}$ of its 1-skeleton.
    We show that two chambers $s$ and $t$ that share a $(d-1)$-cell $f$ have opposite $\psi$-orientations.
    Indeed, let $u$ be the vertex of $s$ not in $f$, and let $v$ be the vertex of $t$ not in $f$.
    Then it must hold that $\psi(u) = \psi(v)$.
    Let $u_0, \cdots, u_{d-1}$ be the vertices of $f$.
    Then the vertex-orderings $u_0, \cdots, u_{d-1}, u$ of $s$, and $u_0, \cdots, u_{d-1}, v$ of $t$ induce opposite orientations, but the vertex-colours are ordered in the same way.
    Therefore, $s$ and $t$ have different $\psi$-orientations.

    (\ref{itm:heawood-2}) $\Rightarrow$ (\ref{itm:heawood-3}).
    Let $G$ be the dual graph of $C$.
    Let $V_{+} \subseteq V(G)$ be the set of all chambers with positive $\psi$-orientation and let $V_{-} \subseteq V(G)$ be the set of all chambers with negative $\psi$-orientation.
    Observe that this defines a bipartition $V(G) = V_+ \dot\cup V_-$.
    By (\ref{itm:heawood-2}), no two chambers $s,t \in V(G)$ in the same part of this bipartition are adjacent in $G$.
    
    (\ref{itm:heawood-3}) $\Rightarrow$ (\ref{itm:heawood-4}).
    Let $f$ be a $(d-2)$-cell in $C$ and let $G$ be the dual graph of $C$.
    Let $c = t_0 f_1 t_1 f_2 t_2 \dots f_{\ell-1} t_{\ell-1} f_{\ell} t_0$ be the cyclic ordering of the $(d-1)$-cells $f_1,\dots,f_{\ell}$ and chambers $t_0,\dots,\allowbreak t_{\ell-1}$ around $f$, which describes a cycle in $G$ (see \autoref{subsec:topology}).
    By (\ref{itm:heawood-3}), $G$ is bipartite and hence $\ell$ is even, showing that $f$ is incident with an even number of $(d-1)$-cells in $C$.
    
    (\ref{itm:heawood-4}) $\Rightarrow$ (\ref{itm:heawood-1}).
    Let $C$ be a triangulation of $\Sbb^d$ and assume that all its $(d-2)$-cells are incident with an even number of $(d-1)$-cells.
    Let $X = \Sigma_{d+1}$ be the symmetric group whose elements are the permutations on $\{0,1,2,\dots,d\}$.
    Recall that for a chamber $t$ with vertices $u_0 < u_1 < \dots < u_d$, each element $\pi_t \in X$ corresponds to a colouring of the vertices of $t$ with $d+1$ colours where vertex~$u_i$ gets colour $\pi_t(i)$.
    
    We will use the \nameref{lem:local-global-colouring} (\autoref{lem:local-global-colouring}) to show the implication (\ref{itm:heawood-4}) $\Rightarrow$ (\ref{itm:heawood-1}).
    For this, we construct a proper canonical local $(d+1)$-colouring of $C$, as defined in \autoref{dfn:locally-colourable}.
    First, we need to construct the bijective maps $g_{\vec e} \colon X \to X$ on the directed edges $\vec e$ of the dual graph $G$.
    Then it suffices to show that these bijective maps satisfy \autoref{dfn:locally-colourable} (\ref{itm:locally-colourable-1}) and (\ref{itm:locally-colourable-2}).
    By \autoref{lem:local-global-colouring}, we then get a proper $(d+1)$-colouring of the 1-skeleton of $C$.
    
    Let $G$ be the dual graph of~$C$.
    We assign bijective maps $g_{\vec e} \colon X \to X$ to the oriented edges $\vec e$ of~$G$, which makes it into a dual $X$-gain graph of $C$, as follows.
    Let $f$ be the $(d-1)$-cell in $C$ corresponding to the edge $\vec e = st$ of $G$.
    For each colouring $(s,\pi_s)$ of $s$ (with $\pi_s \in X$) there exists a unique colouring $(t,\pi_t)$ of $t$ (with $\pi_t \in X$) that agrees with $(s,\pi_s)$ on the vertices of $f$, and vice versa.
    We define $g_{\vec e}$ to be the bijection that maps each colouring $\pi_s \in X$ to the according $\pi_t$.
    Observe that $g_{\cev e} = g_{\vec e}^{-1}$ for all directed edges $\vec e \in E(G)$.

    We prove that \autoref{dfn:locally-colourable} (\ref{itm:locally-colourable-1}) is fulfilled in the following claim.
    
    \begin{claim}
        \label{claim:heawood-face-cycles}
        For each $(d-2)$-cell $f$ in $C$, the $f$-cycle $c$ in $G$ is balanced.
    \end{claim}
    \begin{claimproof}
        Let $c = s_0 \vec e_1 s_1 \vec e_2 s_2 \dots s_{k-1} \vec e_k s_0$ be the $f$-cycle in $G$ and $f_i$ be the $(d-1)$-cell in $C$ corresponding to the edge $\vec e_i$ in $G$, for $i=1,\dots,k$.
        We have to show that $g_c = g_{\vec e_k} \circ \dots \circ g_{\vec e_2} \circ g_{\vec e_1} = \unit$.
        
        By definition of $g_{\vec e_i}$, the colouring $(s_{i-1},\pi_{s_{i-1}})$ and the colouring $(s_i, \pi_{s_i} = g_{\vec e_i}(\pi_{s_{i-1}}))$ in $C$ agree on the vertices of $f_i$ and therefore on the vertices of $f$.
        Let $(s_0,\pi_{s_0})$ be an arbitrary colouring of~$s_0$.
        Define $\pi_{s_i} := (g_{\vec e_i} \circ \dots \circ g_{\vec e_1})(\pi_{s_0})$, which gives colourings $(s_i, \pi_{s_i})$ of $s_i$.
        Then all colourings $(s_i, \pi_{s_i})$ agree on the vertices of the $(d-2)$-cell $f$.
        Without loss of generality, $f$ is coloured with the colours $\{2,3,\dots,d\}$.
        For $i=1,\dots,k$, let $w_i$ be the vertex of $f_i$ that is not a vertex of $f$.
        Then the $w_i$ are coloured alternatingly with the colours $0$ and $1$.
        By~(\ref{itm:heawood-4}), every $f$-cycle has even length (i.e., $k$ is even) and it follows that $g_c(\pi_{s_0}) = \pi_{s_0}$ for every colouring $\pi_{s_0}$ of the vertices of~$s_0$.
        Therefore, $g_c = \unit$ for every $f$-cycle $c$ in $G$.
    \end{claimproof}
    
    Note that \autoref{dfn:locally-colourable} (\ref{itm:locally-colourable-1}) is fulfilled by \autoref{claim:heawood-face-cycles}, and \autoref{dfn:locally-colourable} (\ref{itm:locally-colourable-2}) is fulfilled by the definition of the functions $g_{\vec e}$.
    We use \autoref{lem:local-global-colouring} with $k=d+1$ to obtain a $(d+1)$-colouring of the $1$-skeleton of $C$.
\end{proof}

\section{Proof of (\ref{itm:main-subdivisions-1}) \texorpdfstring{$\Leftrightarrow$}{<=>} (\ref{itm:main-subdivisions-2}) of \autoref{thm:subdivisions}}
\label{sec:subdivisions}

Let $C$ be a triangulation of $\Sbb^d$ and $t$ be a chamber of~$C$.
The $d$-dimensional simplicial complex~$C'$ that is obtained from~$C$ by \emph{subdividing} $t$ is defined as follows.
We add a vertex (i.e.\ a $0$-cell)~$v$ in the chamber $t$, and for each $i=1,2,\dots,d-1$ and for each $(i-1)$-cell~$f$ of~$C$, we add an $i$-cell~$f'$ consisting of~$v$ and the vertices of~$f$.
Observe that $C'$ is again a triangulation of $\Sbb^d$.
For triangulations $C$ and $C'$ of $\Sbb^d$, we say that $C'$ is a \emph{subdivision} of $C$ if there exists a subset $T$ of chambers of $C$ such that $C'$ is obtained from $C$ by subdividing every chamber of $T$.
We say that a triangulation $C$ of $\Sbb^d$ is \emph{subdividable} if there exists a subdivision $C'$ of $C$ such that for each $(d-2)$-cell $f$ in $C'$, the number of incident $(d-1)$-cells of $f$ in $C'$ is divisible by three.

\begin{dfn}[Orientation of properly $(d+2)$-coloured chambers]
    Let $C$ be a triangulation of~$\Sbb^d$ with a proper $(d+2)$-colouring $\psi \colon V \to \{0,\dots,d+1\}$ of its 1-skeleton.
    For a chamber~$t$ of~$C$, let $v_0,\cdots,v_d$ be a positive vertex-ordering.
    Then the $\psi$-orientation of $t$ is \emph{positive} if 
    \[
        \begin{pmatrix}0 & 1 & \cdots & d & d+1\\ \psi(v_0) & \psi(v_1) & \cdots & \psi(v_d) & c\end{pmatrix} \quad \text{is an even permutation,}
    \]
    where $c$ is the colour not used by $v_0,\dots,v_d$, i.e.\ $\{c\} = \{0,\dots,d+1\} \setminus \psi(\{v_0,\dots, v_d\})$, and otherwise \emph{negative}.
\end{dfn}

In order to prove (\ref{itm:main-subdivisions-1}) $\Leftrightarrow$ (\ref{itm:main-subdivisions-2}) of \autoref{thm:subdivisions}, we show the following strengthening.
\begin{thm}
    Let $C$ be a triangulation of $\Sbb^d$.
    Then the following statements are equivalent.
    \begin{enumerate}[(1)]
        \item\label{itm:subdivisions-1} The 1-skeleton of $C$ is $(d+2)$-colourable.
        \item\label{itm:subdivisions-2} There exists a subdivision $C'$ of $C$ and a proper $(d+2)$-colouring $\psi'$ of the 1-skeleton of $C'$ such that all chambers of $C'$ have the same $\psi'$-orientation.
        \item\label{itm:subdivisions-3} There exists a subdivision of $C$ such that for all $(d-2)$-cells, the number of incident $(d-1)$-cells is divisible by three.
    \end{enumerate}
\end{thm}
\begin{proof}
    (\ref{itm:subdivisions-1}) $\Rightarrow$ (\ref{itm:subdivisions-2}).
    Let $C$ be a triangulation of $\Sbb^d$ with a proper $(d+2)$-colouring $\psi$.
    Let $G$ be the dual graph of $C$ and $V_+,V_- \subseteq V(G)$ be the set of all chambers of $C$ with positive (respectively negative) $\psi$-orientation.
    Observe that $V_+$ and $V_-$ partition the set of all chambers, i.e.\ $V(G) = V_+ \,\dot\cup\, V_-$.
    We subdivide all chambers in $V_-$ to obtain the triangulation $C' \supseteq C$ of $\Sbb^d$.
    Let $\psi'$ be the proper $(d+2)$-colouring of $C'$ with $\psi'(v)=\psi(v)$ for all vertices $v \in V(C)$, i.e.\ the (unique) natural extension of the colouring $\psi$ to the triangulation $C'$.
    It remains to show that all chambers have positive $\psi'$-orientation.

    If $t$ is a chamber of $C$ that has positive $\psi$-orientation, then $t$ is still a chamber of $C'$ with a positive $\psi'$-orientation.
    So let $t$ be a chamber of $C$ with negative $\psi$-orientation.
    Let $u$ be the vertex that is added in the inside of $t$ when subdividing~$t$.
    Let $t'$ be a chamber of $C'$ contained in~$t$.
    Then $t'$ consists of the vertices of a $(d-1)$-cell $f$ (say, with vertices $v_0 < v_1 < \dots < v_{d-1}$) of $t$ and the vertex $u$.
    Let $v$ be the vertex of the chamber $t$ that is not a vertex of the chamber $t'$.
    First of all, observe that the vertex-ordering $v_0 < v_1 < \cdots < v_{d-1} < v$ induces the same orientation on the chamber $t$ as the vertex-ordering $v_0 < v_1 < \cdots < v_{d-1} < u$ does on the chamber $t'$.
    Moreover, note that $\psi'(u) \neq \psi'(v)$ since the vertices $u$ and $v$ are adjacent in the 1-skeleton of $C'$.
    Therefore, the permutations 
    \[
        \left(
        \begin{smallmatrix} 
            0 & \cdots & d-1 & d & d+1 \\
            \psi'(v_0) & \cdots & \psi'(v_{d-1}) & \psi'(v) & \psi'(u)
        \end{smallmatrix}\right) \text{ and }
        \left(
        \begin{smallmatrix}
            0 & \cdots & d-1 & d & d+1 \\
            \psi'(v_0) & \cdots & \psi'(v_{d-1}) & \psi'(u) & \psi'(v)
        \end{smallmatrix}\right)
    \]
    have opposite parities.
    Since the $\psi$-orientation of the chamber $t$ is negative, the $\psi'$-orientation of the chamber $t'$ is positive.
    This shows that all chambers of $C'$ have positive $\psi'$-orientation.
    
    (\ref{itm:subdivisions-2}) $\Rightarrow$ (\ref{itm:subdivisions-3}).
    Let $C'$ be a subdivision of $C$ with a proper $(d+2)$-colouring $\psi'$ of the 1-skeleton of $C'$ such that all chambers have the same $\psi'$-orientation (say, positive $\psi'$-orientation).
    We show that for every $(d-2)$-cell $f$ in $C'$, the number of incident $(d-1)$-cells is divisible by three.

    Let $c = t_0 f_1 t_1 f_2 t_2 \dots t_{k-1} f_k t_0$ be the cyclic ordering of the $(d-1)$-cells $f_1,\dots,f_k$ and chambers $t_0,t_1,\dots,t_{k-1}$ around $f$ induced by the embedding of $C'$ in $\mathbb{S}^d$.
    Let $w_i$ be the vertex of $f_i$ that is not in $f$, for $i=1,\dots,k$.
    Let $c_1,c_2,c_3$ be the three colours that are not colours of $f'$.
    Without loss of generality, $\psi'(w_1) = c_1$ and $\psi'(w_2) = c_2$.
    Since all chambers $t_0,t_1,\dots,t_{k-1}$ have the same $\psi'$-orientation, it must hold that $\psi'(w_3) = c_3$, $\psi'(w_4) = c_1$, $\dots$, $\psi'(w_k) = c_3$, $\psi'(w_1) = c_1$.
    Then it follows that $k$ is divisible by three and hence the number of $(d-1)$-cells that are incident with~$f'$ is divisible by three.
    
    (\ref{itm:subdivisions-3}) $\Rightarrow$ (\ref{itm:subdivisions-1}).
    Let $C'$ be a subdivision of $C$ such that for every $(d-2)$-cell, the number of incident $(d-1)$-cells is divisible by three.
    It suffices to show that there exists a $(d+2)$-colouring of the $1$-skeleton of $C'$.
    We fix an arbitrary ordering $v_0 < v_1 < \dots < v_n$ of the vertices $V = \{v_0,v_1,\dots,v_n\}$ of $C'$.
    Let $X = \Sigma_{d+2}$ be the symmetric group symmetric group whose elements are the permutations on $\{0,1,2,\dots,d+1\}$.
    For a chamber $t$ with vertices $u_0 < u_1 < \dots < u_d$ and an element $\pi_t \in X$, the colouring $(t,\pi_t)$ of $t$ induces a colouring of the vertices of $t$, i.e.\ vertex $u_i$ gets colour $\pi_t(i)$ (and no vertex of $t$ gets colour $\pi_t(d+1)$).

    We will use the \nameref{lem:local-global-colouring} (\autoref{lem:local-global-colouring}) to show the implication (\ref{itm:subdivisions-3}) $\Rightarrow$ (\ref{itm:subdivisions-1}).
    For this, we construct a proper canonical local $(d+1)$-colouring of $C$, as defined in \autoref{dfn:locally-colourable}.
    First, we need to construct the bijective maps $g_{\vec e} \colon X \to X$ on the directed edges $\vec e$ of the dual graph $G$.
    Then it suffices to show that these bijective maps satisfy \autoref{dfn:locally-colourable} (\ref{itm:locally-colourable-1}) and (\ref{itm:locally-colourable-2}).
    By \autoref{lem:local-global-colouring}, we then get a proper $(d+1)$-colouring of the 1-skeleton of $C$.
    
    Let $f$ be the $(d-1)$-cell in $C'$ corresponding to the edge $\vec e = st \in E(G)$ in the dual graph~$G'$ of~$C'$.
    For each colouring $(s,\pi_s)$ of $s$ (with $\pi_s \in X$) there exists a unique colouring $(t,\pi_t)$ of $t$ that agrees with $(s,\pi_s)$ on the vertices of $f$ and satisfies $\pi_s(d+1) \neq \pi_t(d+1)$; and vice versa.
    We define $g_{\vec e}$ to be the bijective map that maps each colouring $\pi_s \in X$ to the according $\pi_t$.
    Observe that $g_{\cev e} = g_{\vec e}^{-1}$ for all directed edges $\vec e \in E(G)$.
    
    \begin{claim}
    \label{claim:5-colouring-face-cycles}
        For each $(d-2)$-cell $f$ in $C'$, the $f$-cycle $c$ in $G'$ is balanced.
    \end{claim}
    \begin{claimproof}
        Let $c = s_0 \vec e_1 s_1 \vec e_2 s_2 \dots s_{k-1} \vec e_k s_0$ be the $f$-cycle in $G$ and $f_i$ be the $(d-1)$-cell in $C$ corresponding to the edge $\vec e_i$ in $G$, for $i=1,\dots,k$.
        We have to show that $g_c = g_{\vec e_k} \circ \dots \circ g_{\vec e_2} \circ g_{\vec e_1} = \unit$.
        
        By definition of $g_{\vec e_i}$, the colouring $(s_{i-1},\pi_{s_{i-1}})$ of $s_{i-1}$ and the colouring $(s_i, \pi_{s_i} = g_{\vec e_i}(\pi_{s_{i-1}}))$ of~$s_i$ agree on the vertices of $f_i$ and therefore on the vertices of~$f$.
        Let $(s_0,\pi_{s_0})$ be an arbitrary colouring of $s_0$.
        Define $\pi_{s_i} := (g_{\vec e_i} \circ \dots \circ g_{\vec e_1})(\pi_{s_0})$, which gives a colouring $(s_i,\pi_{s_i})$ of $s_i$.
        Then all colourings $(s_i,\pi_{s_i})$ agree on the vertices of the $(d-2)$-cell $f$.
        Without loss of generality, $f$ is coloured with the colours $\{3,4,\dots,d+1\}$.
        For $i=1,\dots,k$, let $w_i$ be the vertex of $f_i$ that is not a vertex of $f$.
        Then the $w_i$ are coloured $0,1,2$ (in that cyclic order) or $0,2,1$ (in that cyclic order).
        Since every $f$-cycle has length divisible by three it follows that $g_c(\pi_{s_0}) = \pi_{s_0}$ for every $\pi_{s_0} \in X$.
        Therefore, $g_c = \unit$ for every $f$-cycle $c$ in $G$.
    \end{claimproof}
    
    Note that  \autoref{dfn:locally-colourable} (\ref{itm:locally-colourable-1}) is fulfilled by \autoref{claim:5-colouring-face-cycles}, and  \autoref{dfn:locally-colourable} (\ref{itm:locally-colourable-2}) is fulfilled by the definition of the functions $g_{\vec e}$.
    We use \autoref{lem:local-global-colouring} with $k=d+2$ to obtain a $(d+2)$-colouring of the $1$-skeleton of $C$.
\end{proof}

\subsection{Nonsubdividable triangulations}

By the four-colour theorem and by \autoref{thm:subdivisions}, every plane triangulation is subdividable, i.e.\ has a subdivision such that every vertex has degree divisible by three.
For $d = 3$ there exist triangulations of $\Sbb^3$ that are not subdividable, for example the triangulation of $\Sbb^3$ from~\cite{ueckerdtpersonalcommunication} whose 1-skeleton is the complete graph $K_6$, see \autoref{thm:Kn-in-3D}.
In the \autoref{thm:counterexample} we show that these triangulations of $\Sbb^3$ can be used to obtain triangulations of $\Sbb^d$ for $d > 4$ whose 1-skeleton contains a complete graph.

\begin{prop}[\cite{ueckerdtpersonalcommunication}]
\label{thm:Kn-in-3D}
    For every $k \geq 4$ there exists a triangulation of $\Sbb^3$ whose 1-skeleton is the complete graph $K_k$.
\end{prop}

\begin{prop}
\label{thm:counterexample}
    Let $d \geq 3$ and $k \geq 1$.
    There exists a triangulation of $\Sbb^d$ whose $1$-skeleton contains the complete graph $K_{d+k}$.
\end{prop}

\begin{proof}
    We show \autoref{thm:counterexample} by induction on $d$.
    For $d=3$ we use \autoref{thm:Kn-in-3D} to obtain a triangulation of $\Sbb^3$ whose 1-skeleton is the complete graph $K_{3+k}$.
    
    Assume that \autoref{thm:counterexample} holds for some $d \geq 3$.
    Let $C'$ be a triangulation of $\Sbb^d$ whose $1$-skeleton contains the complete graph $K_{d+k}$ as a subgraph.
    Consider an embedding of $C'$ in $\Sbb^{d+1}$ and observe that this cell complex has two chambers (i.e.\ the `inside' and `outside').
    We add the vertex $x$ in the inside and the vertex $y$ in the outside of $C'$.
    Then we add for every $d'$-cell~$f'$ of~$C'$ two $(d'+1)$-cells consisting of the vertices of $f'$ and $x$ (respectively $y$).
    This construction is also called `double cone', which is a triangulation of $\Sbb^{d+1}$.
    Since the $1$-skeleton of $C'$ contains the complete graph $K_{d+k}$ as a subgraph, the $1$-skeleton of $C$ contains the complete graph $K_{(d+1)+k}$ on the vertices of $K_{d+k}$ and $x$ as a subgraph.
\end{proof}

The following corollary follows from \autoref{thm:counterexample} and \autoref{thm:subdivisions}.
\begin{cor}
    For every $d \geq 3$ there exists a triangulation of $\Sbb^d$ that is not subdividable.\qed
\end{cor}

\subsection{Complexity}
The problem of deciding whether the 1-skeleton of a given triangulation of~$\Sbb^d$ is $(d+1)$-colourable is in the complexity class $P$. In fact, by the result of Joswig~\cite{Joswig_2002} (see \autoref{thm:heawood}), an algorithm only needs to check whether every $(d-2)$-cell is incident with an even number of $(d-1)$-cells.

Let us consider the problem of deciding $(d+2)$-colourability of the 1-skeleton of a given triangulation of $\Sbb^d$.
For the case $d=2$ and by the work on the four-colour theorem, a quadratic time algorithm for four-colouring planar graphs has been developed~\cite{Robertson_1996}. For every $d \geq 3$ however, we show that the problem of deciding whether the 1-skeleton of a given triangulation of $\Sbb^d$ is $(d+2)$-colourable is $NP$-complete.
As an intermediate step, we prove the following lemma.

\begin{lem}
\label{lem:complexity-5-colouring}
    It is $NP$-complete to decide whether the 1-skeleton of a given triangulation of $\Sbb^3$ (without parallel edges) is $5$-colourable.
\end{lem}
\begin{proof}
    By \cite{Garey_1976}, it is $NP$-complete to decide whether a given planar 2-connected graph is 3-colourable. We reduce this problem to the problem of deciding whether the 1-skeleton of a given triangulation of $\Sbb^3$ is $5$-colourable.

    Let $G=(V,E)$ be a planar 2-connected graph together with an embedding of $G$ in $\Sbb^2$ (which can be determined in linear time~\cite{Hopcroft_1974}). Let $F$ be the set of faces determined by the embedding of~$G$. Since $G$ is 2-connected, every face in $F$ is bounded by a cycle.

    From $G$ we first build a simplicial 2-complex $C$ as follows.
    First, we add two vertices~$v_1$ and~$v_2$. Then we add edges $uv_i$ for all $u \in V(G)$ and $i \in \{1,2\}$. Moreover, we add for every edge $uu' \in E(G)$ two 2-cells $uu'v_1$ and $uu'v_2$. Let $C$ be the resulting 2-complex.
    Then $C$ can be embedded in $\Sbb^3$ as follows.
    Let $\iota \colon G \hookrightarrow \Sbb^2$ be an embedding of $G$ into $\Sbb^2$, and let $\iota' \colon \Sbb^2 \hookrightarrow \Sbb^3$ be an obvious embedding of $\Sbb^2$ into $\Sbb^3$. Then $\Sbb^3 \setminus \iota'(\Sbb^2)$ determines two path-connected components, which we will refer to as the `inside' and the `outside' of the embedding of $\Sbb^2$ into $\Sbb^3$. First, we embed $G$ in $\Sbb^3$ using $\iota' \circ \iota$. Then we embed $v_1$ and every edge and every 2-cell incident with $v_1$ into the inside, and we embed $v_2$ and every edge and every 2-cell incident with $v_2$ into the outside of the embedding of $\Sbb^2$ into $\Sbb^3$.

    For a face $f \in F$, let $s_f$ be the unique chamber of $C$ with $\iota'(f) \subseteq s_f$.
    Note that the map that maps each $f \in F$ to the chamber $s_f$ of~$C$ is a bijection between the set $F$ and the set of chambers of $C$. 
    Moreover, observe that for every $f \in F$, the chamber $s_f$ is bounded by the vertices $v_1,v_2$ and the vertices on the boundary of $f$ in~$G$.

    For each face $f \in F$, whose boundary consists of the vertices $u_1 u_2 \ldots u_k$ in that cyclic order say, we do all of the following. First, we add two vertices $x_f$ and $y_f$. Moreover, we add all edges $u_i x_f$ and $u_i y_f$ for $i \in [k]$, and the edges $v_1x_f, x_fy_f$ and $y_fv_2$. Now, we add the 2-cells on the vertices $u_iu_{i+1}x_f, u_iu_{i+1}y_f, v_1x_fu_i, x_fy_fu_i$ and $y_fv_2u_i$ for every $i \in [k]$. This yields a 2-complex~$C'$. Note that the embedding of $C$ in $\Sbb^3$ can be extended to an embedding of $C'$ in $\Sbb^3$ by placing each $x_f$ and $y_f$ and all edges and 2-cells incident with $x_f$ or $y_f$ into the chamber~$s_f$ of~$C$.
    In fact, it is straightforward to check that each chamber of $C'$ is a 3-cell, so $C'$ is a triangulation of $\Sbb^3$. Observe that $C'$ does not have parallel edges. Moreover, the triangulation~$C'$ can be constructed from~$G$ in polynomial time.

    It remains to show that $G$ is 3-colourable if and only if the 1-skeleton of $C'$ is 5-colourable. If~$G$ has a proper 3-colouring $g \colon V(G) \to [3]$, we extend $g$ to a proper 5-colouring $c \colon V(C') \to [5]$ of the 1-skeleton of $C'$ by letting $c(v_1):= c(y_f):=4$ for every $f \in F$, and $c(x_f):=c(v_2):=5$ for every $f \in F$.
    Conversely, let $c \colon V(C') \to [5]$ be a proper 5-colouring of the 1-skeleton of $C'$. If~$G$ is bipartite then we are done. Otherwise there is a face $f \in F$ on an odd number of vertices. So the vertices on the boundary of~$f$ receive three distinct colours by $g$. Note that $x_f$ and $y_f$ are adjacent to $v_1$ and $v_2$, respectively, and to all vertices on the boundary of~$f$. So if $c(v_1)=c(v_2)$ then $c(x_f)=c(y_f)$, a contradiction. Therefore, say $4 = c(v_1) \neq c(v_2) = 5$. Since $v_1$ and $v_2$ are both adjacent to every vertex of $G$ in $C'$, the vertices of $G$ are coloured with $1,2$ and~3 only.
\end{proof}

\begin{prop}
    Let $d \geq 3$ be an integer. It is $NP$-complete to decide whether the 1-skeleton of a given triangulation of $\Sbb^d$ is $(d+2)$-colourable.
\end{prop}
\begin{proof}
    We proceed by induction on $d$.
    For $d=3$ the statement follows from \autoref{lem:complexity-5-colouring}.
    Assume that for some $d \geq 3$, it is $NP$-complete to decide whether a given triangulation of $\Sbb^d$ is $(d+2)$-colourable. Let $C$ be a triangulation of $\Sbb^d$ and $C'$ be the double-cone of $C$, i.e.\ $C'$ is a triangulation of $\Sbb^{d+1}$ with $V(C') = V(C) \,\dot\cup\, \{x,y\}$ and $E(C') = E(C) \cup \{xu,yu \mid u \in V(C)\}$. The same construction is used in \autoref{thm:counterexample}. Then $C$ is $(d+2)$-colourable if and only if $C'$ is $((d+1)+2)$-colourable.
\end{proof}

\section{Proof of the `moreover'-part of \autoref{thm:subdivisions}}
\label{sec:2-edge-colouring}

Given a graph $G$ and an edge-colouring $\psi$, we say that a triangle in $G$ is \emph{monochromatic} if all of its edges have the same colour in $\psi$.
Let $R_k(3)$ be the smallest integer $n \in \Nbb$ such that every $k$-edge-colouring of $K_n$ contains a monochromatic triangle. It is known that $R_2(3)=6$ and $R_3(3)=17$.

\begin{lem}
    Every $(R_k(3)-1)$-colourable graph has a $k$-edge-colouring without monochromatic triangles.
\end{lem}
\begin{proof}
    Let $n:=R_k(3)-1$.
    Let $G$ be a graph with a proper $n$-colouring $\varphi \colon V(G) \to [n]$. Let $\psi \colon E(K_n) \to [k]$ be a $k$-edge-colouring of $K_n$ without monochromatic triangles.
    We define a $k$-edge-colouring $\psi' \colon E(G) \to [k]$ by assigning to the edge $e=uv$ the colour $\psi'(e):=\psi(\{\varphi(u),\varphi(v)\})$.

    We need to show that every triangle $u_1 e_1 u_2 e_2 u_3 e_3 u_1$ has two edges of distinct colours.
    Since~$\varphi$ is a proper $n$-colouring of $G$, we have that $\varphi(u_1) \varphi(u_2) \varphi(u_3)$ is a triangle in $K_n$. Since $\psi$ is an edge-colouring of $K_n$ without monochromatic triangles, we have without loss of generality $\psi(\{\varphi(u_1), \varphi(u_2)\}) \neq \psi(\{\varphi(u_2), \varphi(u_3)\})$. Hence $\psi'(e_1) \neq \psi'(e_2)$.
\end{proof}
\begin{cor}
\label{cor:edge-colouring-without-monochromatic-faces}
    Every triangulation of $\Sbb^3$ whose 1-skeleton is $(R_k(3)-1)$-colourable has a $k$-edge-colouring without monochromatic faces.\hfill$\square$
\end{cor}

We prove the `moreover'-part of \autoref{thm:subdivisions} by showing the following stronger result.
\begin{thm}
    Let $C$ be a triangulation of $\Sbb^3$. Then the following statements are equivalent.
    \begin{enumerate}[(1)]
        \item\label{itm:2-edge-colouring-1} The 1-skeleton of $C$ is $5$-colourable.
        \item\label{itm:2-edge-colouring-2} Every subdivision of $C$ has a 2-edge-colouring without monochromatic\footnote{A face is \emph{monochromatic} if all edges on its boundary have the same colour.} faces.
        \item\label{itm:2-edge-colouring-3} The maximal subdivision $C'$ of $C$ (that is obtained from $C$ by subdividing every chamber) has a 2-edge-colouring without monochromatic faces.
        \item\label{itm:2-edge-colouring-4} The triangulation $C$ has a 2-edge-colouring such that on the boundary of each chamber, every colour forms a path of length 3.
        \item\label{itm:2-edge-colouring-5} There exists a subdivision $C'$ of $C$ such that for every edge in $C'$, the number of incident faces is divisible by three.
    \end{enumerate}
\end{thm}
\begin{proof}
    (\ref{itm:2-edge-colouring-1}) $\Rightarrow$ (\ref{itm:2-edge-colouring-2}). Let $C$ be a triangulation of $\Sbb^3$ whose 1-skeleton is 5-colourable. Then the 1-skeleton of any subdivision $C'$ of $C$ is again 5-colourable. Since $R_2(3)=6$, we can construct a 2-edge-colouring of $C'$ without monochromatic faces by \autoref{cor:edge-colouring-without-monochromatic-faces}.

    (\ref{itm:2-edge-colouring-2}) $\Rightarrow$ (\ref{itm:2-edge-colouring-3}) is trivial.

    (\ref{itm:2-edge-colouring-3}) $\Rightarrow$ (\ref{itm:2-edge-colouring-4}). Let $C$ be a triangulation of $\Sbb^3$ and let $C'$ be maximal subdivision of $C$. Let $\psi \colon E(C') \to \{\text{red},\text{blue}\}$ be a 2-edge-colouring of $C'$ without monochromatic faces. Obviously, $\psi|_{E(C)}$ is a 2-edge-colouring of $C$ without monochromatic faces.
    Let $t$ be a chamber of $C$ on the vertices $u_1,u_2,u_3,u_4$.
    Let $E(t)$ be the set of edges on the boundary of the chamber $t$, i.e. $E(t):=\{u_i u_j \mid 1 \leq i < j \leq 4\}$.
    It suffices to show that both colours appear exactly three times in $E(t)$.

    We assume for contradiction that one color, say red, appears less than three times in $E(t)$.
    If red appears in at most one edge of $E(t)$, then one can find a monochromatic face on the boundary of $t$, a contradiction.
    So, red appears twice and blue appears four times in $E(t)$. Since we do not have monochromatic faces on the boundary of $t$, the red edges in $E(t)$ must form a matching, say $\psi(u_1u_2)=\psi(u_3u_4)=\text{red}$. Let $v$ be the vertex of $C'$ that was added when subdividing $t$. Since the edges of the triangle $u_1u_2v$ must contain both colours and since $\psi(u_1u_2)=\text{red}$, we have without loss of generality $\psi(u_1v)=\text{blue}$. Since the triangles $u_1 u_4 v$ and $u_1 u_3 v$ must contain both colours and since we have $\psi(u_1 u_4)=\psi(u_1 v)=\psi(u_1 u_3) = \text{blue}$, we must have $\psi(u_4 v)=\psi(u_3 v)=\text{red}$. But then, the triangle $u_3 u_4 v$ is monochromatic, a contradiction.

    (\ref{itm:2-edge-colouring-4}) $\Rightarrow$ (\ref{itm:2-edge-colouring-5}).
    Let $\psi \colon E(C) \to \{\text{red},\text{blue}\}$ be a 2-edge-colouring such that on the boundary of each chamber, every color forms a path.
    We say that a chamber $t$ is \emph{even} if the order of the vertices along the red path on the boundary of $t$ induces a positive orientation of $t$; otherwise we say that~$t$ is \emph{odd}.
    Let $C'$ be the triangulation of $\Sbb^3$ that is obtained from $C$ by subdividing every odd chamber. We show that for every edge in $C'$, the number of incident faces is divisible by three.

    First, observe that there exists a (unique) way to extend the 2-edge-colouring of $C$ to a 2-edge-colouring $\psi' \colon E(C') \to \{\text{red},\text{blue}\}$ of $C'$ such that on the boundary of each chamber, every color forms a path. Indeed, let $v_0v_1v_2v_3$ be an order of the vertices of a subdivided chamber~$t$ along the red path and let~$v$ be the vertex that is added when subdividing~$t$. Then we put $\psi'(vv_0)=\psi'(vv_3)=\text{red}$ and $\psi'(vv_1)=\psi'(vv_2)=\text{blue}$. Then it is straightforward to check that for each of the four chambers of $C'$ included in $t$, every colour forms a path of length 3 on the boundary of the respective chamber.
    Moreover, the following fact is also straightforward to check, see \autoref{fig:edge-colouring-subdivision}.
    \begin{figure}
        \centering
        \includegraphics[width=0.4\textwidth]{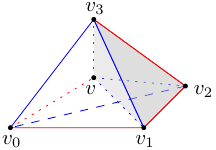}
        \caption{If the chamber $t$ of $C$ on the red path $v_0 v_1 v_2 v_3$ is odd, then the chamber $t'$ of $C'$ on the red path $v_1 v_2 v_3 v$ is even.}
        \label{fig:edge-colouring-subdivision}
    \end{figure}
    \begin{equation}
        \label{eq:all-chambers-even}
        \text{All chambers of $C'$ are even with respect to the 2-edge-colouring $\psi'$.}
    \end{equation}
    Now, we show that for every edge $e$ in $C'$, the number of incident faces is divisible by three. Let $c=t_0 f_1 t_1 f_2 t_2 \dots t_{k-1} f_k t_0$ be the cyclic ordering of the faces $f_1,\dots,f_k$ and chambers $t_0,t_1,\dots,t_{k-1}$ around $e$ induced by the embedding of $C'$ in $\Sbb^3$. Let $w_i$ be the vertex of $f_i$ that is not in $e$, for $i=1,\dots,k$. Let $u$ and $v$ be the endvertices of $e$ such that for all $i=1,\dots,k$, the ordering $u,v,w_i,w_{i+1}$ induces a positive orientation of the corresponding chamber.
    We distinguish two cases.

    Assume that the edge $e$ is red in $\psi'$. If $\psi'(uw_i)=\psi'(vw_i)=\text{blue}$, then we have $\psi'(uw_{i+1})=\text{red}$ and $\psi'(vw_{i+1})=\text{blue}$ since an ordering of the vertices along the red path (here $w_i w_{i+1} u v$) induces a positive orientation on the corresponding chamber by~\ref{eq:all-chambers-even}. Analogously, if $\psi'(uw_{i})=\text{red}$ and $\psi'(vw_{i})=\text{blue}$ then $\psi'(uw_{i+1})=\text{blue}$ and $\psi'(vw_{i+1})=\text{red}$. Analogously, if $\psi'(uw_{i})=\text{blue}$ and $\psi'(vw_{i})=\text{red}$ then $\psi'(uw_{i+1})=\psi'(vw_{i+1})=\text{blue}$. Therefore, if $e$ is red then the number of incident faces is divisible by three.

    Now, assume that the edge $e$ is blue in $\psi'$. 
    This case is similar to the previous case.
    If $\psi'(uw_i)=\psi'(vw_i)=\text{red}$, then we have $\psi'(uw_{i+1})=\text{red}$ and $\psi'(vw_{i+1})=\text{blue}$ since an ordering of the vertices along the red path (here $v w_i u w_{i+1}$) induces a positive orientation on the corresponding chamber by~\ref{eq:all-chambers-even}. Analogously, if $\psi'(uw_{i})=\text{red}$ and $\psi'(vw_{i})=\text{blue}$ then $\psi'(uw_{i+1})=\text{blue}$ and $\psi'(vw_{i+1})=\text{red}$. Analogously, if $\psi'(uw_{i})=\text{blue}$ and $\psi'(vw_{i})=\text{red}$ then $\psi'(uw_{i+1})=\psi'(vw_{i+1})=\text{red}$. Therefore, if $e$ is blue then the number of incident faces is divisible by three.

    (\ref{itm:2-edge-colouring-5}) $\Rightarrow$ (\ref{itm:2-edge-colouring-1}) follows from \autoref{thm:subdivisions}.
\end{proof}

\begin{cor}
\label{cor:complete-subdivision-K6}
    Let $C$ be a triangulation of $\Sbb^3$ whose 1-skeleton is the complete graph $K_6$, and let $C'$ be the maximal subdivision of $C$. Then every 2-edge-colouring of $C'$ has a monochromatic face.\hfill$\square$
\end{cor}

We complement \autoref{cor:complete-subdivision-K6} by showing that every simplicial complex embedded in $\Sbb^3$ has a 4-edge-colouring without monochromatic faces. Additionally, we ask whether any triangulation of $\Sbb^3$ has a 3-edge-colouring without monochromatic faces, see \autoref{que:3}.

Recall that the \emph{link graph} of a simplicial 2-complex $C$ at a vertex $v$ is the graph $L$ whose vertices are the edges incident to $v$ in $C$, and two vertices $e_1,e_2 \in E(C)$ share an edge in $L$ if $e_1,e_2$ viewed as edges in $C$ share a face in $C$. If the simplicial complex $C$ can be embedded in $\Sbb^3$, then the link graph at each vertex of $C$ is planar.

\begin{prop}
\label{prop:4-edge-colouring}
    Let $C$ be a simplicial 2-complex embedded in $\Sbb^3$. Then $C$ has a 4-edge-colouring without monochromatic faces.
\end{prop}
\begin{proof}
    We prove this proposition by induction on the number of vertices $n$ of $C$.
    If $n=1$, then the statement trivially holds.
    Hence, let $n \geq 2$ and let $v$ be a vertex of $C$. Let $C-v$ be the simplicial complex where we remove all cells (i.e.\ vertices, edges and faces) that contain $v$. Then, $C-v$ has a 4-edge-colouring without monochromatic faces by the induction hypothesis. 

    Consider the link graph $L$ of $C$ at $v$. Since $L$ is planar, it has a 4-colouring $\varphi$ by the four-colour theorem~\cite{Appel_1989}. We extend the 4-edge-colouring of $C-v$ to $C$ by giving every edge $e$ that contains the vertex $v$ the colour $\varphi(e)$.

    We need to show that $C$ does not contain a monochromatic face. If the face $f$ does not contain the vertex $v$, it is also a face in $C-v$ and therefore not monochromatic. Hence, assume that the face $f$ contains the vertex $v$, i.e.\ $f$ is bounded by the edges $e_1,e_2,e_3$ such that $e_1$ and $e_2$ are incident to $v$. Then, $e_1$ and $e_2$ are adjacent vertices in the link graph $L$ and therefore receive distinct colours.
\end{proof}

\section{Edge and face colourings}
\label{sec:edge-colourings}

In this section, we show how to obtain an edge colouring from a vertex colouring of a simplicial complex.
An edge-colouring of a simplicial complex $C$ is \emph{proper} if for every face (i.e.\ $2$-cell) $f$, every two distinct edges $e_1$ and $e_2$ adjacent in $f$ receive distinct colours.

In order to derive a proper edge-colouring from a vertex-colouring, we use the following fact about 1-factorizations of complete graphs.
A \emph{$1$-factorization} of a graph $G$ is a partition of its edge set into perfect matchings.
It is well-known that for every positive integer $n$, the complete graph~$K_{2n}$ has a $1$-factorization, see~\cite{mendelsohnrosa} for example.

\begin{prop}
\label{prop:vertex-coloring-implies-edge-coloring}
    Let $k$ be a positive integer and $C$ be a simplicial complex whose $1$-skeleton has a proper $2k$-colouring.
    Then $C$ can be edge-coloured with $2k-1$ colours.
\end{prop}
\begin{proof}
    Let $\varphi \colon V(C) \to [2k]$ be a $2k$-colouring of the $1$-skeleton of $C$.
    Let $E(K_{2n}) = M_1 \dot\cup M_2 \dot\cup \dots \dot\cup\allowbreak M_{2k-1}$ be a $1$-factorization of the complete graph $K_{2n}$ on the vertex set $[2k]$.
    To each edge~$uv$ of~$C$ we assign the colour~$i$ with $\{\varphi(u),\varphi(v)\} \in M_i$.
    Since $\varphi(u) \neq \varphi(v)$ for all edges $uv$ in $C$, this is a well-defined edge-colouring.
    We denote this edge-colouring by the function $\psi \colon E(C) \to [2k-1]$.
    We need to show that every face $f=uvw$ satisfies $\psi(uv) \neq \psi(vw) \neq \psi(wu) \neq \psi(uv)$.
    Suppose for contradiction that two edges of $f$ have the same colour, without loss of generality $\psi(uv) = \psi(vw)$.
    Then $\{\varphi(u), \varphi(v)\}$ and $\{\varphi(v), \varphi(w)\}$ are in the same perfect matching $M_i$ of the $1$-factorization of $K_{2n}$.
    Therefore $\varphi(u) = \varphi(w)$, a contradiction since there exists an edge $uw$ in $C$.
\end{proof}

For $k=2$, the converse of~\autoref{prop:vertex-coloring-implies-edge-coloring} holds for triangulations $C$ of $\Sbb^3$, as shown by Carmesin, Nevinson and Saunders~\cite{carmesin2022characterisation}, i.e.\ $C$ has a proper 3-edge-colouring if and only if its 1-skeleton has a proper 4-colouring.
However, the converse of \autoref{prop:vertex-coloring-implies-edge-coloring} does not hold in general, even for triangulations of $\Sbb^3$.
By~\cite{ueckerdtpersonalcommunication}, for every $n \geq 4$ there exists a triangulation of $\Sbb^3$ whose 1-skeleton is the complete graph $K_n$, but which is 12-edge-colourable by~\cite{Kurkofka_2025}.
In particular, $13$-edge-colourability of $C$ (in this case $k=7$) does not imply any upper bound on the chromatic number of the 1-skeleton of $C$.

\begin{cor}
    Every triangulation of $\Sbb^d$ such that every $(d-2)$-cell is incident with an even number of $(d-1)$-cells is edge-colourable with $d$ colours if $d$ is odd, and with $d+1$ colours if $d$ is even.
\end{cor}
\begin{proof}
    Follows from \autoref{thm:heawood} and \autoref{prop:vertex-coloring-implies-edge-coloring}.
\end{proof}
\begin{cor}
    Every subdividable triangulation of $\Sbb^d$ is edge-colourable with $d+1$ colours if $d$ is even, and with $d+2$ colours if $d$ is odd.
\end{cor}
\begin{proof}
    Follows from \autoref{thm:subdivisions} and \autoref{prop:vertex-coloring-implies-edge-coloring}.
\end{proof}

For triangulations of $\Sbb^3$, we can also derive a proper face-colouring from a vertex-colouring as follows.
A face-colouring of a $2$-complex embedded in $\Sbb^3$ is \emph{proper} if for every chamber, every two distinct faces on the boundary of that chamber receive different colours.
The \emph{face-chromatic number} of a triangulation $C$ of $\Sbb^3$ is the least positive integer $k$ such that there exists a proper $k$-face-colouring of $C$.
The face-chromatic number of a triangulation of $\Sbb^3$ is at least 4 since every chamber (i.e.\ tetrahedron) is incident with four faces.
For a trivial upper bound, consider the graph $G=(F,E)$ where $F$ is the set of faces of $C$ and two faces are adjacent in $G$ if they are contained in the boundary of the same chamber.
This graph has maximum degree $6$ which proves that the face-chromatic number of a triangulation of $\Sbb^3$ is at most $7$.

\begin{prop}
    Let $C$ be a triangulation of $\Sbb^3$ whose $1$-skeleton has a $5$-colouring.
    Then $C$ can be face-coloured with five colours.
    On the other hand, there exists a triangulation of $\Sbb^3$ that has no $4$-face-colouring but whose 1-skeleton has a $5$-colouring.
\end{prop}
\begin{proof}
    Let $\varphi \colon V(C) \to \Zbb_5$ be a $5$-colouring of the $1$-skeleton of $C$.
    Let $E(K_5) = M_1 \,\dot\cup\,\allowbreak M_2 \,\dot\cup\, \dots \,\dot\cup\,\allowbreak M_5$ be a partition of the edge set of the complete graph $K_5$  (on the vertex set $\Zbb_5$) into maximum matchings
    (take a $1$-factorization of $K_6$ and remove a vertex).
    To each face $f$ with vertices $v_0, v_1, v_2 \in f$, we assign the colour $i$ with $\Zbb_5 \setminus \{\varphi(v_0), \varphi(v_1), \varphi(v_2)\} \in M_i$.
    We denote this face-colouring by the function $\psi \colon F(C) \to \Zbb_5$.
    Consider two faces $f_1$ and $f_2$ on the boundary of the same chamber, and with vertices $v_0,v_1,v_2 \in f_1$ and $v_1,v_2,v_3 \in f_2$.
    Assume that $\psi(f_1) = \psi(f_2)$, then $\{\varphi(v_0), \varphi(v_1), \varphi(v_2)\} = \{\varphi(v_1), \varphi(v_2), \varphi(v_3)\}$ by the definition of $\psi$.
    It follows that $\varphi(v_0) = \varphi(v_3)$, which is a contradiction since $v_0$ and $v_3$ are adjacent in the $1$-skeleton of $C$.

    Consider the triangulation $C$ of $\Sbb^3$ whose 1-skeleton is the complete graph $K_5$, i.e.\ $C$ is obtained from the tetrahedron by subdividing one chamber.
    Obviously, the 1-skeleton of $C$ is 5-colourable.
    On the other hand observe that the triangles of $K_5$ are the faces of $C$.
    Hence $C$ has $\binom{5}{3} = 10$ faces.
    Assume that $C$ is 4-face-colourable.
    Then there exist three faces $f_1,f_2,f_3$ in $C$ that have the same colour.
    But then, two of them share an edge and therefore lie on the boundary of the same chamber of $C$ (since every edge in $C$ is incident with exactly three faces), a contradiction.
\end{proof}
\begin{cor}
    Every subdividable triangulation of $\Sbb^3$ can be face-coloured with five colours and this is best possible.
\end{cor}
\begin{proof}
    Follows from \autoref{thm:subdivisions} and \autoref{prop:vertex-coloring-implies-edge-coloring}.
\end{proof}

\section{Outlook}
\label{sec:outlook}

\autoref{que:1} remains open for $k \geq 3$:

\begin{customque}{\ref{que:1}}
    For $k,d \geq 3$, can you find a structural characterisation of the triangulations of $\Sbb^d$ whose $1$-skeleton is $(d+k)$-colourable?
\end{customque}


By \autoref{prop:4-edge-colouring}, every triangulation of $\Sbb^3$ has a 4-edge-colouring without monochromatic faces. By \autoref{cor:complete-subdivision-K6} there exists a triangulation of $\Sbb^3$ such that every 2-edge-colouring has a monochromatic face.
\begin{que}
    \label{que:3}
    Is there a triangulation of $\Sbb^3$ such that every 3-edge-colouring has a monochromatic face?
\end{que}%
Note that, by \autoref{cor:edge-colouring-without-monochromatic-faces} and since $R_3(3)=17$, every triangulation of $\Sbb^3$ whose 1-skeleton has a proper 16-colouring also has a 3-edge-colouring without monochromatic faces.

\section*{Acknowledgements}
I am grateful to Johannes Carmesin and Jan Kurkofka for stimulating discussions.

\bibliographystyle{plain}
\bibliography{literatur}

\end{document}